\documentclass[journal]{IEEEtran}

\usepackage{lipsum}
\usepackage{amsfonts}
\usepackage{amsmath}
\usepackage{amsthm}
\usepackage{amssymb}
\usepackage{graphicx}
\usepackage{epstopdf}
\usepackage{verbatim}
\usepackage{algorithmic}
\usepackage{algorithm}
\usepackage{enumitem}
\usepackage{color}
\newtheorem{assumption}{Assumption}
\newtheorem{lemma}{Lemma}
\newtheorem{theorem}{Theorem}

\begin{document}

\title{Distributed Stochastic Nonsmooth Nonconvex Optimization}

\author{Vyacheslav Kungurtsev 
  \thanks{Kungurtsev is with the Department of Computer Science, Faculty of Electrical Engineering, Czech Technical University in Prague, Prague, Czech Republic. His work was supported by
  the OP VVV project
CZ.02.1.01/0.0/0.0/16 019/0000765 ``Research Center for Informatics''}%
}

\maketitle

\begin{abstract}
Distributed consensus optimization has received considerable attention in recent years; several distributed consensus-based  algorithms have been proposed for (nonsmooth) convex and (smooth) nonconvex objective functions. However, the behavior of these distributed algorithms on {\it nonconvex, nonsmooth and stochastic} objective functions is not understood. This class of functions and distributed setting are motivated by several applications, including problems in machine learning and signal processing. 

 This paper
presents the first convergence analysis of the   decentralized stochastic subgradient method for 
such classes of problems, over networks modeled as undirected, fixed, graphs.



\end{abstract}


\section{Introduction}
We consider the following  nonsmooth unconstrained nonconvex optimization problem over a network of $n$ agents:
\begin{equation}\label{eq:P}
\min_{{\boldsymbol{\theta}}\in \mathbb{R}^m}\, F(\boldsymbol{\theta}):= \sum_{i=1}^n f_i(\boldsymbol{\theta}),
\end{equation}
where 
$f_i:\mathbb{R}^m\to\mathbb{R}$ is the cost function of agent $i$,   known only to agent $i$. We make no assumptions about the smoothness or convexity of  $F$; each $f_i$ is only assumed to be locally Lipschitz continuous, and thus in general it is  nonconvex and nonsmooth.  Furthermore, we assume that agent $i$ does not have access to the (elements of the)   subgradient of its own $f_i$, but only unbiased stochastic estimates of the elements of the 
subgradient are available. 
 Agents are
connected through a communication network, modeled as a  connected, undirected  graph. No specific topology  is assumed for the graph (such as star or hierarchical structure). In this setting, agents seek  to cooperatively solve
Problem ~(\ref{eq:P})   by exchanging information with their immediate neighbors in the network.

 This class of problems and distributed setting arises naturally from many applications in different fields, including  signal processing, statistical data
analysis, machine learning, and engineering. For instance data may be collected and stored   across different nodes and networks;
and loss functions, regularizers, or risk measures that are nonsmooth are increasingly utilized
in statistical data analysis~\cite{loh2017support,rockafellar2000optimization}.
Alternatively, consider training Deep Neural Net architectures on data existing at different centers that may
communicate across a network, or using distributed memory parallel architectures with high latency. 
These problems often involve large volumes of data and result in a loss function that is the 
finite sum of typically nonsmooth functions, due to the presence of rectified linear units, max-pooling, and
other activations, or nonsmooth loss functions~\cite{goodfellow2016deep}. 
Clusters incorporating CPU cores each with its
own distributed memory are common tools available to solve such problems~\cite{leighton2014introduction}. 
Each worker (e.g., core) has access to its own storage of memory, and can communicate data to other workers as needed. \vspace{-0.2cm}

\subsection{Related works}
We are not aware of any result on  the convergence of decentralized schemes for stochastic, nonsmooth, nonconvex
problems in the form (\ref{eq:P}).  
There is a vast literature on distributed algorithms for deterministic (nonsmooth) convex problems; see, e.g., the tutorial papers  \cite{nedic-procIEEE} \cite[Ch. 2 \& 3]{book-CIME}, the earlier works  \cite{Nedic-Ozdaglar-Parrilo_TAC10,Nedic-Olshevsky_TAC2015},  and references therein.   Distributed methods for  {\it nonconvex}    optimization  have also received   attention     \cite{NEXT,scutari-sun19,DoF,pmlr-v70-hong17a,WYin_ncvxDGD_TSP}. The schemes in \cite{DoF,pmlr-v70-hong17a,WYin_ncvxDGD_TSP} are applicable to unconstrained {\it smooth} nonconvex optimization, with  \cite{DoF} handling also compact constraints while   \cite{NEXT,scutari-sun19} can handle   objectives  with additive   {\it nonsmooth convex}  functions. Distributed algorithms for {\it stochastic} optimization problems over networks were proposed in
\cite{ram2010distributed,bianchi2013convergence,Bianchi_TIT_2013,Tatarenko2015}; 
we group these papers as follows. The work \cite{ram2010distributed}   studied  the effects of stochastic subgradient errors on the convergence of   the distributed gradient projection algorithm  \cite{Nedic-Ozdaglar-Parrilo_TAC10} applied to {\it convex, smooth}, constrained optimization over undirected graphs.  
 A distributed projected stochastic gradient algorithm (resp. distributed stochastic approximation algorithm) involving random gossip between agents and decreasing stepsize was studied in \cite{bianchi2013convergence} for {\it nonconvex, smooth}, constrained optimization (resp. \cite{Bianchi_TIT_2013}); to deal with the nonconvexity of the objective, the analysis in \cite{bianchi2013convergence} relies    on  stochastic approximation techniques introduced in \cite{benaim2005stochastic}.   Finally, \cite{Tatarenko2015}  studied the effect of additive i.i.d. noise to the iterates of the  push-sum gradient algorithm \cite{Nedic-Olshevsky_TAC2015} applied to (deterministic ) {\it nonconvex, smooth} optimization over digraphs.   \vspace{-0.3cm}

\subsection{Contributions}
In this work, we introduce the first provably convergent distributed stochastic subgradient method   solving Problem~(\ref{eq:P}), over undirected graphs.  The proposed algorithm can be considered as an extension 
of~\cite{bianchi2013convergence}, in presenting the same setting of stochastic approximation for modeling
the sequence of iterates, however, with the objective function
not assemed to be continuously differentiable.

\section{Assumptions and Preliminaries}

\subsection{System model}

We will assume that $F(\boldsymbol{\theta})$ is continuous and subdifferentially
regular~\cite{rockafellar2009variational}. We shall refer to the subgradient operator $\partial f(\cdot)$
of any regular function $f(\cdot)$ as the Clarke subgradient~\cite{clarke1990optimization}, defined,
at a point $\bar x$, as the convex hull of,
\[
\left\{\frac{f(x_k)-f(\bar x)}{x_k-\bar x},\, x_k\to \bar x\right\}
\]
We note that by Rademacher's Theorem~\cite{rockafellar2009variational} it holds that a subdifferentially
regular function is continuously differentiable almost everywhere. Thus it can be said that $F(\boldsymbol{\theta})$
is equal to a selection of one of a possibly infinite set of continuously differentiable functions.

The communication network of the agent is modeled as a fixed undirected graph $\mathcal{G}:= (\mathcal{V},\mathcal{E})$ with 
vertices $\mathcal{V}\triangleq \{1,..,I\}$ and $\mathcal{E}:=\{(i,j)|i,j\in\mathcal{V}\}$ representing the agents and
communication links, respectively. We assume that the graph $\mathcal{G}$ is strongly connected. 

Each agent $i$ has access to and controls an estimate of the primal variables $\mathbf{x}_{(i)}^\nu$. 
We define the graph matrix $\mathbf{L}=\mathbf{I}-\mathbf{W}$ where $\mathbf{W}=\mathbf{A}\otimes \mathbf{I}$ with
$\mathbf{A}$ satisfying $\mathbf{A}_{ij}>0$ for $i\neq j$ 
if $(i,j)\in\mathcal{E}$ and $\mathbf{A}_{ij}=0$ otherwise.

We assume that $\mathbf{L}$ is double stochastic. The eigenvalues of $\mathbf{L}$ are real and can be sorted in a
nonincreasing order $1=\lambda_1(\mathbf{L})\ge \lambda_2(\mathbf{L})\ge ...\ge \lambda_n(\mathbf{L})\ge -1$. 

Defining,
\[
\beta\triangleq \max\{|\lambda_2(\mathbf{L})|,|\lambda_n(\mathbf{L})|\}
\]
we shall make the following assumption,
\begin{assumption}\label{as:beta}
It holds that,
\[
\beta<1
\]
\end{assumption}

In addition, we assume that each agent $i$ does not have access to the entire subgradient of its function,
i.e., $\partial f_i(\mathbf{x}_{(i)})$, but only has access to a stochastic oracle estimating some
element of this set. In particular, we assume the following regarding any noisy subgradient estimates
$\mathbf{y}_{(i)}$ evaluated at $\mathbf{x}_{(i)}$,
\begin{assumption}\label{as:noise}
Each agent $i$ can has access to an oracle that returns $\mathbf{y}_{(i)}$ which may be written as,
\[
\mathbf{y}_{(i)}= g_{(i)}+\delta M_{(i)}
\]
where $g_{(i)}\in \partial f_i(\mathbf{x}_{(i)})$ and $\delta M_{(i)}$ is a Martingale difference
stochastic noise, and,
\begin{itemize}
\item $\mathbb{E}\left[\delta M_{(i)}\right]=0$
\item $\mathbb{E}\left[\|\delta M_{(i)}\|^2\right]\le R$
\item For every realization $\|\mathbf{y}^\nu_{(i)}\|\le B$.
\end{itemize}
where $R,B\in\mathbb{R}^+$ are some constants.
\end{assumption}

Finally we make an assumption about the structure of the points of nonsmoothness. In particular,
we consider that each $f_i$ is defined to be the maximum of a set of smooth functions. Furthermore
the set of activities, i.e., the \emph{active} smooth function corresponding to the value of $f_i(\cdot)$
at $x$ does not significantly change across $x$ in neighborhoods of arbitrarily small size for almost all $\mathbf{x}$. 
It can be seen
that this assumption holds for the standard problems arising in estimation and data science.

\begin{assumption} \label{as:nonsmoothactive}
Each $f_i(\cdot)$ can be defined as,
\begin{equation}\label{eq:defnactive}
f_i(\mathbf{x}) = \max_{j\in \mathcal{C}_i} f_{i,j} (\mathbf{x}).
\end{equation}
It holds that $f_{i,j}$ has Lipschitz continuous first derivatives, and the Lipschitz constants
across all $i$ and $j$ are bounded uniformly by $L$.

Define $\mathcal{A}_{i}(\mathbf{x}):= \{j\in\mathcal{C}_i:f_i(\mathbf{x})=f_{i,j}(\mathbf{x})$.
For each $i$ and every $\mathbf{x}$, either,
\begin{itemize}
\item $\forall \mathbf{x}_j\to \mathbf{x}\text{ with }\mathbf{x}_j\neq \mathbf{x}$, $\mathcal{A}_{i}(\mathbf{x}) \neq \mathcal{A}_{i}(\mathbf{x}_j)$,
or,
\item $\exists D$ such that for all $\mathbf{x}$, for all $\mathbf{z}\in \mathcal{B}_o(\mathbf{x},D)$, 
it holds that $\mathcal{A}_{i}(\mathbf{x})=\mathcal{A}_{i}(\mathbf{z})$, where $\mathcal{B}_o(\mathbf{x},D)$ is the open ball centered at $\mathbf{x}$ with radius $D$.
\end{itemize}
The assumption implies, in particular that there exists a set $\mathcal{Z}_i$ of zero measure with respect to
$\mathbb{R}^n$ such that all the points satisfying the first condition are contained in $\mathcal{Z}_i$.
\end{assumption}

\subsection{Some Examples}
Consider training a deep neural network, which results in an objective function that is a composition of nested 
functions and activations, with a sum additive loss function at the final exterior, 
with training data $\mathbf{z}$, e.g.,, 
\[
\begin{array}{l}
F(\boldsymbol{\theta}) = l(\phi(\boldsymbol{\theta},\mathbf{z}),\mathbf{z}),\\
\phi(\boldsymbol{\theta},\mathbf{z}) = \phi_1(\phi_2
(\phi_4(\boldsymbol{\theta},\mathbf{z}),\phi_5(\boldsymbol{\theta},\mathbf{z})),\phi_3(\boldsymbol{\theta},\mathbf{z}))
\end{array}\]
where, for instance, $l$ could be an $l1$ loss, e.g., 
$l(\phi(\boldsymbol{\theta},\mathbf{z}),\mathbf{z}) = \|\phi(\boldsymbol{\theta},\mathbf{z})-\mathbf{z}\|_1$,
$\phi_3$ and $\phi_4$ could be sigmoids, i.e., $\phi_j(\boldsymbol{\theta},\mathbf{z})=\frac{1}{1+e^{-[\boldsymbol{\theta}]_J\cdot[\mathbf{z}]_J}}$, where
we use the subscript $[\cdot]_J$ to indicate the components in the index set $J$ of the vector inside,
$\phi_5$ a Rectified Linear Unit, i.e., $\phi_5(x) = \max(0,[\boldsymbol{\theta}]_J\cdot[\mathbf{z}]_J)$
and $\phi_2(\phi_3,\phi_4)=\max(\phi_3,\phi_4)$. Notice the function is summable, but clearly
nonconvex and nonsmooth, and also non-separable in variables (thus presenting no viable closed form prox solution).

Other examples of nonconvex nonsmooth functions can be found in, e.g.,~\cite{davis2019stochastic}. They include
robust phase retrieval, covariance matrix estimation, blind devonvolution, sparse PCA and conditional
value at risk. Note that all but the last one are immediately given as a sum of functions across 
data, thus if data is distributed across a network the setting applies. Conditional value at risk,
if evaluated with sample average approximation, with the data on the different samples distributed,
also becomes a summable distributed optimization problem.

\section{Preliminaries and Algorithm}
Define $\mathbf{x}$ to be the stack of vectors $\{\mathbf{x}_{(i)}\}$ and problem,
\begin{equation}
\min_{\mathbf{x} \in \mathbb{R}^{mn}} \, F_d(\mathbf{x})= \sum_{i=1}^n f_i(\mathbf{x}_{(i)}), 
\end{equation}
to be an auxillary optimization problem to facilitate the analysis of solving problem~(\ref{eq:P}).




We present the Algorithm for this paper as Algorithm~\ref{alg}. 
  \begin{algorithm}[t]
	\caption{Stochastic Distributed Optimization}
	\label{alg}
	\begin{algorithmic}
\STATE{\textbf{Initialization:}} $\mathbf{x}^0_{(1)}=\mathbf{x}^0_{(2)}=...=\mathbf{x}^0_{(n)}\in\mathbb{R}^m$. Set $\nu=0$.
		\WHILE{a termination criterion is not met, each agent does:}
\STATE{Obtain a noisy subgradient estimate $\mathbf{y}^\nu_{(i)}\approx g^\nu_{(i)}\in \partial f_i(\mathbf{x}^\nu_{(i)})$}
\STATE{Let,
\begin{equation}\label{eq:algstep}
\mathbf{x}^{\nu+1}_{(i)} = (1-w_{ii})\left(\mathbf{x}^\nu_{(i)}-\gamma^\nu \mathbf{y}^\nu_{(i)}\right)+\sum_{j\in\mathcal{N}_i} w_{ij}\left(\mathbf{x}^\nu_{(j)}-\gamma^\nu \mathbf{y}^\nu_{(i)}\right),
\end{equation}
}
\STATE{$\nu=\nu+1$}
\ENDWHILE
		\RETURN $\mathbf{x}^k$
	\end{algorithmic}
\end{algorithm}

The primary step of the algorithm, given by~(\ref{eq:algstep}) can be also given as,
\begin{equation}\label{eq:algstepbig}
\mathbf{x}^{\nu+1} =  \left((\mathbf{I}_n-\mathbf{W}) \otimes \mathbf{I}_m\right)\left(\mathbf{x}^\nu-\gamma^\nu  \mathbf{y}^\nu\right).
\end{equation}

We make the following assumption on the step size,
\begin{assumption}\label{as:stepsize}
The stepsize sequence $\{\gamma^\nu\}$ satisfies,
\begin{enumerate}
\item $\sum_{\nu=0}^\infty \gamma^\nu = \infty$
\item $\sum_{\nu=0}^\infty \left(\gamma^\nu\right)^2 < \infty$
\end{enumerate}
\end{assumption}

The proof is structured as follows,
\begin{enumerate}
    \item First we shall show that with probability one, the algorithm achieves
    consensus, in particular, each agent's estimate of the iterates approaches the
    mean of the estimates. This result is the same as in~\cite{bianchi2013convergence}.
    \item Next we define a differential inclusion (DI) whose equilibrium points correspond to
    stationary points of~\eqref{eq:P}. We show that that the mean iterate follows
    a stochastic process defined as a particular perturbed stochastic approximation to 
    the flow defined by this differential inclusion. 
    \item Using a result in~\cite{kushner2003stochastic}, we conclude that this approximation
    converges to an invariant set of the DI. 
    \item Finally, applying recent results relating invariant sets of DIs to local minimizers
    of corresponding nonsmooth optimization problems, we conclude that the mean of the iterates
    converges to a stationary point of ~\eqref{eq:P}
\end{enumerate}

We use the theory of \emph{stochastic approximation and perturbed differential inclusions}
as developed in~\cite{kushner2003stochastic,borkar2009stochastic} and refined for nonsmooth problems
in~\cite{benaim2005stochastic}.

We shall define the following terminology, arising in, for instance~\cite{benaim2005stochastic}.

Consider a differential inclusion,
\begin{equation}\label{eq:adia}
\mathbf{x}(t)\in -G(\mathbf{x}(t))
\end{equation}

A set $A$ is said to be \emph{internally chain transitive} if for any two elements 
$\mathbf{z}_1,\mathbf{z}_2\in A$ and any $\epsilon>0$ and $T>0$, there exists an integer $n\in \mathbb{N}$,
solutions $\mathbf{x}_1,...,\mathbf{x}_n$ to~\eqref{eq:adia} and $t_1,...,t_n>T$ with a) $\mathbf{x}_i(s)\in A$
for all $0\le s\le t_i$, all $i\in\{1,...,n\}$, b) $\|x_i(t_i)-x_{i+1}(0)\|\le \epsilon$ for all 
$i\in\{1,...,n-1\}$, and c) $\|\mathbf{x}_1(0)-\mathbf{z}_1\|\le \epsilon$ and $\|\mathbf{x}_i(t_i)-\mathbf{z}_2\|\le \epsilon$.




\section{Convergence Analysis}
We define the mean iterate to be,
\[
\bar{\mathbf{x}}^k = (\mathbf{1}_n\mathbf{1}_n^T\otimes \mathbf{I}_m  ) \mathbf{x}^k
\]

We first present a necessary standing assumption for this section.

\begin{assumption}\label{as:bound}
For every agent $i$, the iterates $\mathbf{x}^{\nu}_{(i)}$ are bounded almost surely. 
\end{assumption}

Alternatively, one can introduce a compact set on which the iterates are constrained to lie in.

\subsection{Consensus}
\begin{lemma}\label{lem:cons}
The iterates reach consensus, i.e., for all $i$,
\[
\lim_{k\to\infty}\left\|\bar{\mathbf{x}}^k-\mathbf{x}^k_{(i)}\right\|=0
\]
\end{lemma}
\begin{proof}
Same as in~\cite[Lemma 1]{bianchi2013convergence}.
\end{proof}

\subsection{Differential Inclusion and Stochastic Approximation}
Let $G(\boldsymbol{\theta})=\partial F(\boldsymbol{\theta})$. The differential flow defined for the sequential subgradient method
for minimizing $F(\boldsymbol{\theta})$ with arbitrarily small stepsizes is given by,
\begin{equation}\label{eq:di}
\boldsymbol{\theta}(t)\in - G(\boldsymbol{\theta})
\end{equation}

The update to $\bar{\mathbf{x}}$ is given by,
\begin{equation}\label{eq:update}
\begin{array}{l}
\bar{\mathbf{x}}^{k+1} = \left(\mathbf{1}\mathbf{1}^T \otimes \mathbf{I}\right)\left((\mathbf{I}-\mathbf{W}) \otimes \mathbf{I}\right)\left(\mathbf{x}^\nu-\gamma^\nu  \mathbf{y}\right) \\
\qquad = \bar{\mathbf{x}}^{k}-\frac{1}{m}\sum_{i=1}^m \left(g_{i,k} (\mathbf{x}^k_{(i)})+\delta M_{i,k}\right)
\end{array}
\end{equation}
where 
\[
g_{i,k}(\mathbf{x}_{(i)})\in \left\{\nabla f_{i,j} (\mathbf{x}_{(i)}):\, j\in \mathcal{A}_i(\mathbf{x}^k_{(i)})\right\}
\]
holds almost surely by Assumption~\ref{as:nonsmoothactive}.

Let us define,
\[
\begin{array}{l}
\bar{\mathbf{Y}}^k = \frac{1}{m}\sum_{i=1}^m \left(g_{i,k} (\mathbf{x}^k_{(i)})+\delta M_{i,k}\right) \\ \qquad = 
\frac{1}{m}\sum_{i=1}^m \left(g_{i,k} (\bar{\mathbf{x}}^k)-g_{i,k} (\bar{\mathbf{x}}^k)+g_{i,k} (\mathbf{x}^k_{(i)})+\delta M_{i,k}\right) \\
\qquad = 
g_k(\bar{\mathbf{x}}^k) +\frac{1}{m}\sum_{i=1}^m \left(g_{i,k} (\mathbf{x}^k_{(i)})-g_{i,k} (\bar{\mathbf{x}}^k)\right)+\delta M_k \\ 
\qquad = 
g_k(\bar{\mathbf{x}}^k) +\beta_k+\delta M_k
\end{array}
\]
Let $m(t)$ be the smallest integer greater than $t$. Let 
\[
M^0(t) = \sum_{k=0}^{m(t)-1} \gamma^k \delta M_k 
\]
and
\[
B^0(t) = \sum_{k=0}^{m(t)-1} \gamma^k \beta_k 
\]

\subsection{Convergence}
We recall the following Theorem, arising as~\cite[Theorem 5.6.3]{kushner2003stochastic}
\begin{theorem}\label{th:kush}
Given a stochastic process,
\[
\begin{array}{l}
\mathbf{X}^{k+1} = \mathbf{X}^k+\gamma^k \mathbf{Y}^k \\
\mathbf{Y}^k = -g_k(\mathbf{X}^k)+\beta_k+\delta M_k
\end{array}
\]
Define $M^0$ and $B^0$ as above.

Assume,
\begin{itemize}
\item $\mathbb{E}[\mathbf{Y}^k]<\infty$
\item 
\[
\begin{array}{l}
\lim_{k}\sup_{j\ge k} \max_{0\le t\le T} \left\|M^0(jT+t)-M^0(jT)\right\|=0,\\ \text{ and,} \\
\lim_{k}\sup_{j\ge k} \max_{0\le t\le T} \left\|B^0(jT+t)-B^0(jT)\right\|=0 \\ \text{ with probability one.} 
\end{array}
\]
\item
\[
\lim_{\Delta\to 0} \limsup_{k} \sup_{m(t_k+\Delta)\ge j\ge k} \frac{|\gamma^j-\gamma^k|}{\gamma^k}=0
\]
\item $g_k(\mathbf{X})$ is continuous, $G(\mathbf{X})$ is upper semicontinuous and,
\[
\lim_{k,j\to \infty} \text{dist} \left\{\frac{1}{j} \sum_{l=k}^{k+j-1} g_l(\mathbf{X}), G(\mathbf{X})\right\} = 0
\]
\item $\mathbf{X}^k$ is bounded with probability one.
\end{itemize}
Then almost surely, limits of convergent subsequences of $\mathbf{X}^k$ are trajectories of the
differential inclusion,
\[
\dot{\mathbf{X}} \in -G(\mathbf{X})
\]
in some bounded internally chain transitive set and $\mathbf{X}^k$ converges to this
set. 
\end{theorem}
\begin{proof}
~\cite[Theorem 5.6.3 and 5.2.1]{kushner2003stochastic}
\end{proof}

We now apply this theorem to the process $\bar{\mathbf{x}}^k$ given by~\eqref{eq:update}.

\begin{theorem}\label{th:limitxbar}
Theorem~\ref{th:kush} applies to $\bar{\mathbf{x}}^k$ for the differential inclusion defined by
~\eqref{eq:adia}, i.e., almost surely, limit points of $\bar{\mathbf{x}}^k$  are trajectories
of~\eqref{eq:adia} and $\bar{\mathbf{x}}^k$ converges to an invariant set of this DI.
\end{theorem}
\begin{proof}
We shall see that the assumptions of Theorem~\ref{th:kush} are satisfied
for this process. 

It holds that $\mathbb{E}[\mathbf{Y}^k]<\infty$ and $\bar{\mathbf{x}}^k$ are bounded with probability one by assumption.
Furthermore,
\[
\lim_{k}\sup_{j\ge k} \max_{0\le t\le T} \left\|M^0(jT+t)-M^0(jT)\right\|=0 \text{ w.p.}1
\] 
by standard
arguments regarding the Martingale difference noise (see, e.g., the proof of~\cite[Theorem 5.2.1]{kushner2003stochastic}). 

Next we have that, by Assumption~\ref{as:nonsmoothactive} and the definition of $g_{i,k}$,
\[
\|g_{i,k} (\mathbf{x}^k_{(i)})-g_{i,k} (\bar{\mathbf{x}}^k)\|\le L \|\mathbf{x}^k_{(i)}-\bar{\mathbf{x}}^k\|\to 0
\]
for all $i$, thus $\beta_k\to 0$ implying, together with the step-size conditions, that
\[
\lim_{k}\sup_{j\ge k} \max_{0\le t\le T} \left\|B^0(jT+t)-B^0(jT)\right\|=0 \text{ w.p.}1 
\] 

Finally, we know that $g_{i,k}(\cdot)$ are continuous and $G(\cdot)$ is upper semicontinuous by definition.

Now, since $\mathbf{x}^k$ is a stochastic process with nonzero noise for all $k$, it holds that there is a set of
dense probability measure such that for all $k>0$, $\mathbf{x}_{(i)}^k\notin \mathcal{Z}$ and $\bar{\mathbf{x}}^k\notin \mathcal{Z}$.

This implies that for all $i$, since $\|\mathbf{x}^k_{(i)}-\bar{\mathbf{x}}^k\|\to 0$, that Assumption~\ref{as:nonsmoothactive} implies,
\[
\lim_{k,j\to\infty} \frac{\sum_{l=k}^{k+j-1} \mathbf{1}(\mathcal{A}_{i}(\bar{\mathbf{x}}^k)\neq \mathcal{A}_{i}(\mathbf{x}_{(i)}^k))}{j}\to 0
\] 
and thus the fourth condition of the Theorem has been shown, and the results follow.
\end{proof}

\subsection{Properties of Limit Points}
The previous sections showed that asymptotically $\bar{\mathbf{x}}^\nu$ are trajectories of the differential
inclusion~\eqref{eq:di}. The proof of~\cite[Theorem 5.2.1]{kushner2003stochastic} concludes that in the case
of the presence of a compact constraint or an ODE instead of a DI, limit points of the sequence are thus
stationary points of~\eqref{eq:P}. In~\cite[Theorem 4.2]{davis2019stochastic} the argument was extended
for an unconstrained nonsmooth function satisfying certain properties. 

\begin{theorem}\cite[Theorem 4.2]{davis2019stochastic}
If it holds that,
\begin{itemize}
\item The set of stationary points of~\eqref{eq:P} is dense, and
\item For any trajectory $\mathbf{z}(t)$ of the DI~\eqref{eq:di}, it holds that if $\mathbf{z}(0)$ is not
stationary, there exists a $T$ such that $F(\mathbf{z}(t))<F(\mathbf{z}(0))$ for $t\in(0,T]$.
\end{itemize}
then every limit point of $\bar{\mathbf{x}}^\nu$ is critical for $F(\cdot)$ and the function
values $F(\bar{\mathbf{x}}^\nu)$ converge.
\end{theorem}

By Assumption~\ref{as:nonsmoothactive} the first condition holds. Second, note that the same Assumption
defines a Whitney $C^1$-stratification of the graph of $F$ and thus by~\cite[Theorem 5.8]{davis2019stochastic}
the second condition holds as well. 

Finally by $\|\mathbf{x}^\nu_{(i)}-\bar{\mathbf{x}}^\nu\|\to 0$ we
have that this convergence theorem holds for the individual iterates as well.

\section{Numerical Results}
We simulated Algorithm~\ref{alg} on training a neural net architecture for the MNIST data set. 
We used a nonsmooth loss function with an $l1$-regularizer, and two layers that included a softmax
and a relu operator, with 100 nodes in the inner layer. Specifically, with $\boldsymbol{\theta} = (w,v,b,c)$ the parameters, 
$A$ the training data and $y$ the labels,
\[
\begin{array}{l}
f_i(\boldsymbol{\theta}) = \left\|\phi(\boldsymbol{\theta},A_i)-y\right\|_1+\lambda\|\boldsymbol{\theta}\|_1 \\
\phi(\boldsymbol{\theta},A)_j = \frac{1}{1+e^{-\sum_k w_{kj} \psi(\boldsymbol{\theta}, A_i)+b_j}}\\
\psi(\boldsymbol{\theta},A)_k = \max\left(0,\sum_l v_{kl} [A_i]_l+c_k\right) 
\end{array}
\]
We ran 8500 iterations using 50 agents with randomly generated connections at 50$\%$ for each potential
edge. Each agent sampled 1$\%$ of its apportioned data set uniformly at each iteration to perform the update.
We use $\alpha=0.1$.

We show the results in Figure~\ref{fig:results}. We see that the iterates appear to be near-consensus.
Furthermore, the norm of the (sub)gradients, evaluated at the average iterate among the agents
is monotonically decreasing, along with the objective value (also evaluated at the average iterate).
Given that diminishing step-sizes must be used, the convergence is slow.

\begin{figure}
\begin{center}
\begin{tabular}{c} 
\includegraphics[scale=0.47]{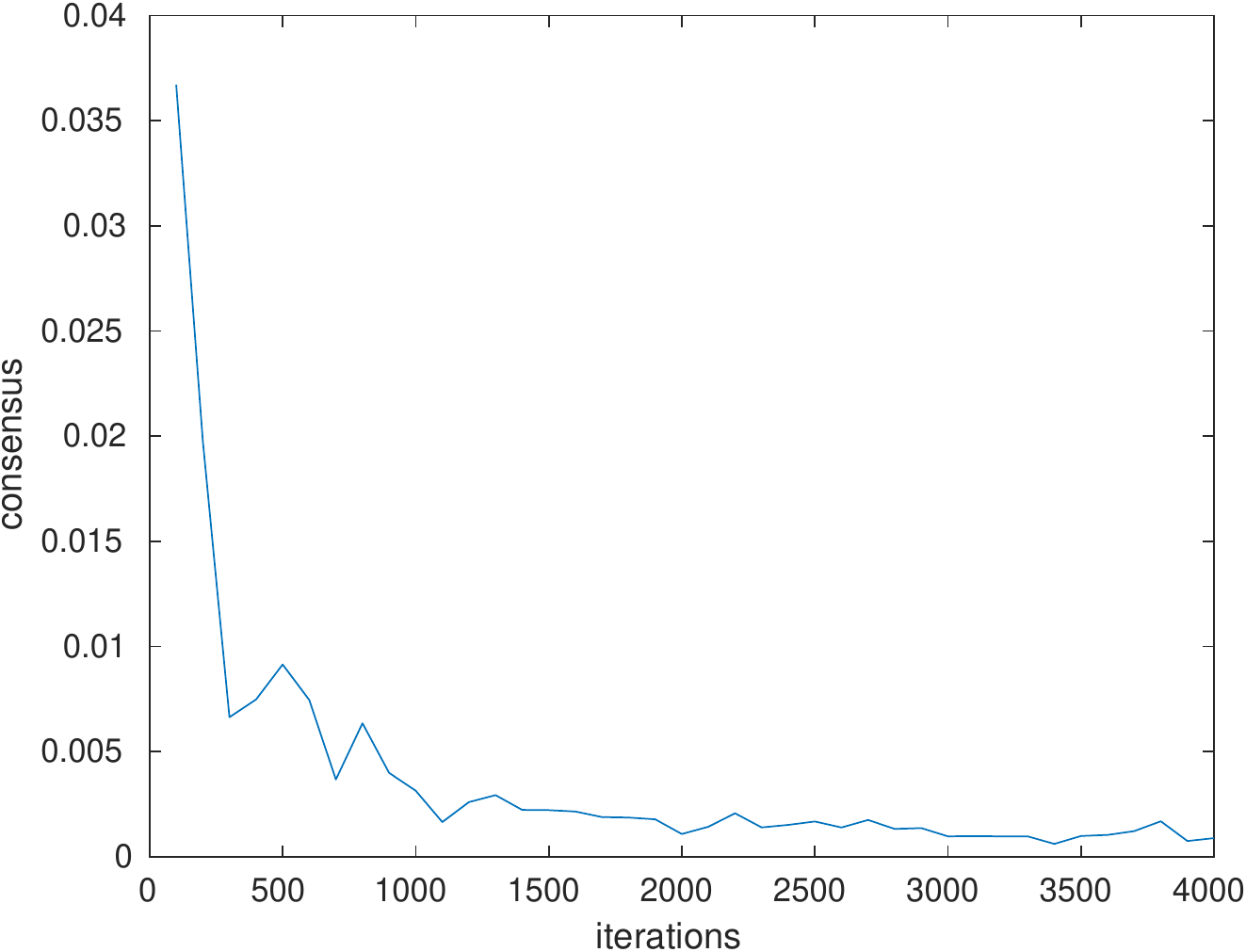}
\\
\includegraphics[scale=0.46]{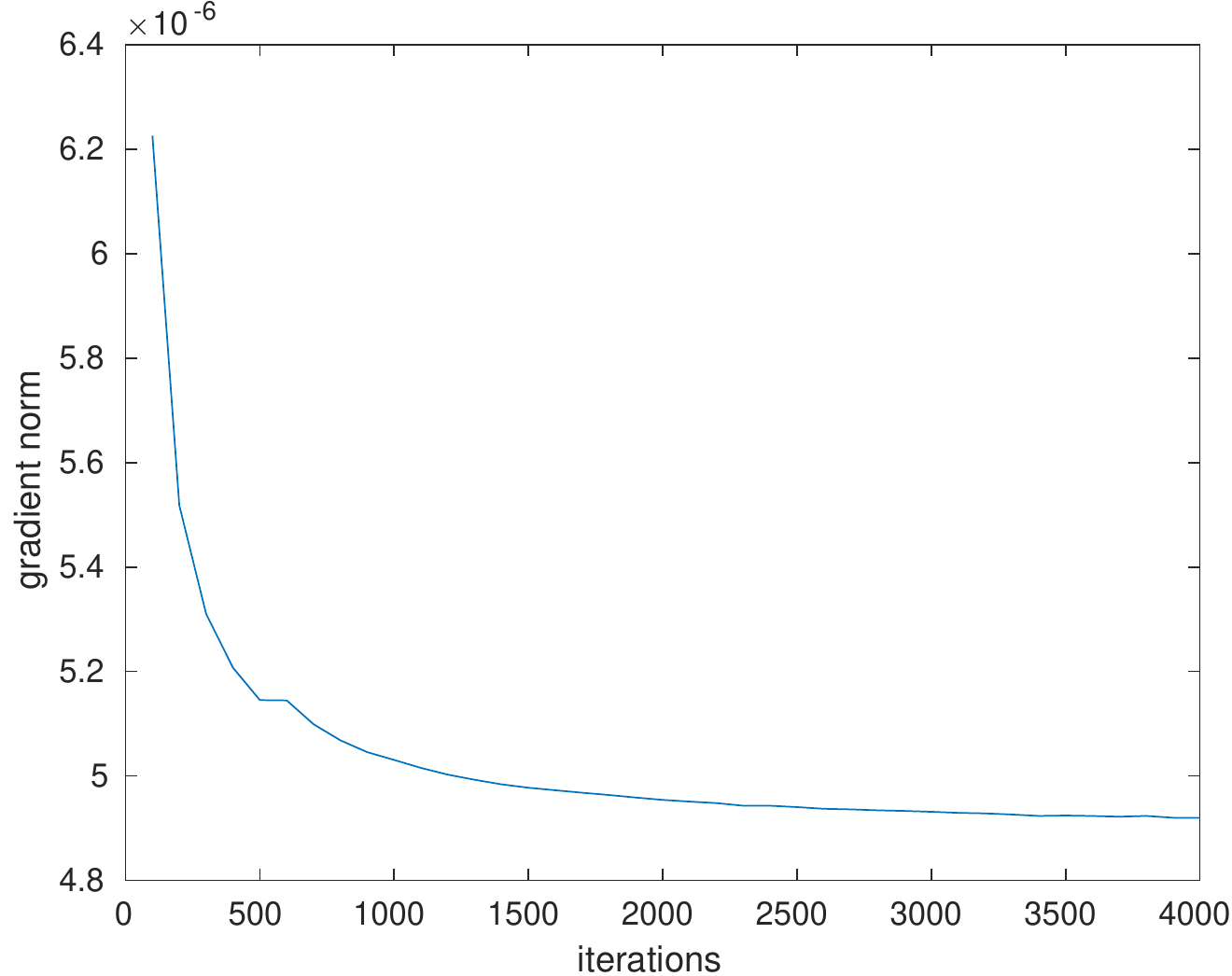}
\\
\includegraphics[scale=0.45]{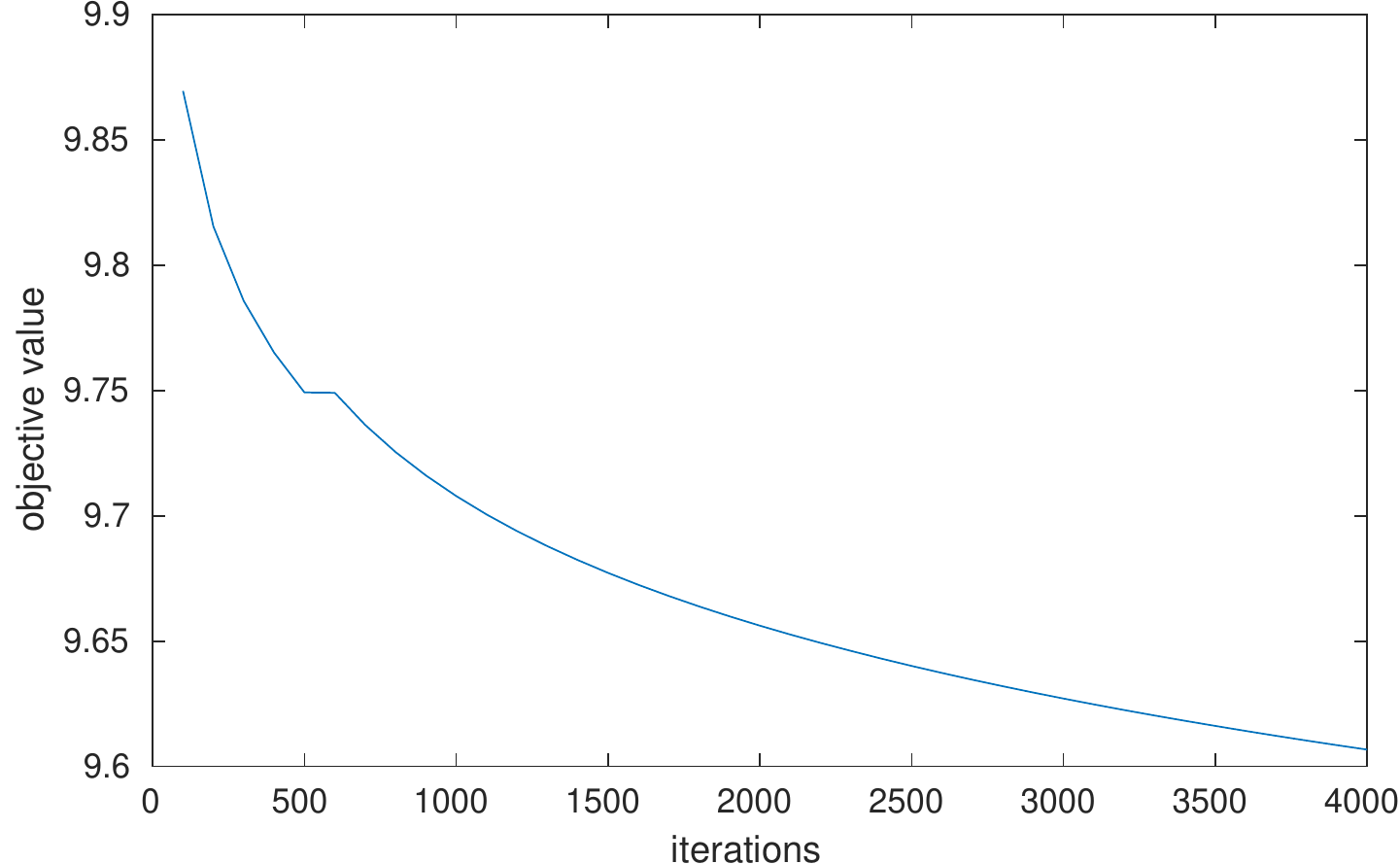}
\end{tabular}
\end{center}
\caption{\label{fig:results} Consensus error, (sub)gradient norm at the average of the iterates, and the objective value per iteration}
\end{figure}

\section{Conclusion}

This paper presents an advancement in the state of the art for analysis of decentralized optimization schemes
in extending the available convergence theory to nonsmooth, nonconvex problems, using stochastic updates.
Using ideas from the control consensus literature and stochastic approximation theory, we were able to
prove convergence of a simple procedure for a standard auxillary problem, and bound the distance
of its solution to a solution of the original problem. We demonstrated the efficacy of the procedure
on a standard example in training DNN architecture. As this just begins the chapter in
the analysis of such problems, there is considerable scope for future research extensions.

\bibliographystyle{IEEEtran}
\bibliography{references}

\end{document}